\documentclass[a4paper,12pt]{article}
\usepackage{amsmath,amssymb, amsfonts, amsthm}
\usepackage{enumerate}
\newtheorem{theorem}{Theorem}
\newtheorem{corollary}[]{Corollary}
\newtheorem{lemma}[]{Lemma}

\newtheorem{proposition}[]{Proposition}

\newtheorem{remark}[]{ Remark}

\renewcommand{\Im}{\mathrm{Im}}
\renewcommand{\Re}{\mathrm{Re}}

\begin{document}

\title{Stark resonances   in   2-dimensional  curved quantum waveguides. }
\author{Philippe Briet \footnote{e-mail: briet@univ-tln.fr}, \\
{\small   Aix-Marseille Universit\'{e}, CNRS, CPT UMR 7332, 13288 Marseille, France, and } \\ 
{ \small Universit\'{e} de Toulon, CNRS, CPT UMR 7332, 83957 La Garde, France}\\
and \\
Mounira Gharsalli  \footnote{e-mail: gharsallimounira@gmail.com}\\
 {\small $^2$Laboratoire EDP, LR03ES04, D\' epartement de Math\' ematiques, Facult\'e des Sciences de Tunis,} \\
{  \small Universit\' e de  Tunis El Manar, El Manar 2092 Tunis, Tunisie}}

\date{}
\maketitle

\begin{abstract}
  In this paper we study  the influence of  an electric field on a  two dimensional  waveguide. We  show that  bound states   that occur under  a   geometrical deformation  of   the guide    turn  into resonances when we apply an electric field   of  small intensity  having a nonzero component  on  the longitudinal direction of the system.

\end{abstract}

 \thanks{}

 \date{}
\noindent
\vskip 1cm

\medskip

 \footnote{MSC-2010 number: 35B34,35P25, 81Q10, 82D77}

\maketitle
\section{Introduction} The study of resonances occurring  in   a quantum system  subjected to a  constant electric field is now a well-known issue among  the mathematical physics community.
In a recent past  a large amount of literature has been devoted to this problem  (see  e.g. \cite{Ha,HpSg} and references therein).  Mostly  these works are concerned with quantum systems living in the whole space $\mathbb R^n$ as e.g. atomic systems \cite {PB2,FK, He,HpVi,Sig,Wa}.
In the present paper we  would like   to address this question for an inhomogeneous quantum system consisting in a curved quantum waveguide in $\mathbb R^2$. It is known that  bound states arise in curved guides \cite{BuGeReSi,DE} and the corresponding eigenfunctions are expected to be  localized in space around  the  deformation.   Therefore,  based on  these results  the  main question   is    what happens with these  bound states   when the electric  field of small intensity  is switched on?

A first  result  is given  in \cite{Exner} where   the  electric field  is supposed to be  orthogonal to  the guide outside a bounded region. But  in   this situation  there is no Stark resonance.

Here we are focusing on a strip $  {\bf \Omega } \subset \mathbb R^2$ of constant width curved within a compact region. The electric field is chosen with a strictly positive component both on the longitudinal direction
of the left part and of the right part of the curved strip.  Roughly speaking this situation  is similar to the one of  an  atomic system interacting with an external  electric field.  Due    to the  field,  an eigenstate of the  curved waveguide  at zero field      turns into  scattering state    which   is able to  escape  at  infinity under the dynamics.    It is then  natural to expect   spectral resonances for this system.
In this work we would like to study this question in the weak field  regime.

The  resonances    are defined as the complex poles  in the second Riemann sheet of  the    meromorphic continuation  of the resolvent  associated  to the   Stark operator. We construct this extension  using   the  distortion theory \cite{AC,WH}.  Our proof of  existence of resonances borrows  elements of   strategy developed  in   \cite{PB2,He}. It is mainly based on  non-trapping estimates     of  \cite{PB2}.  For  the applicability of these techniques  to our model, the   difficulty   we have to solve    is that   the  system   has    a  bounded transverse direction. 
  
  To end this section let us mention  a still open  question   related to   this problem and that  we hope to solve  in  a future work.  We claim that our  regularity assumptions  on the curvature  imply  that  the corresponding  Stark operator  (see \eqref{ST})  has   no real eigenvalue  \cite {AvHe}. In that case the complex poles have a non zero imaginary part then they are resonances in the strict sense of the term  \cite{reed simonIV}. 
 
 Let us briefly review the content of the paper. In section 2 we describe precisely the system, assumptions and the main results. The distortion and the definition of resonances are given respectively in section 3 and 4. In section 5 we prove the existence of resonances. Finally the section 6 is devoted to get an exponential estimate on the width of resonances. Actually we show that      the  imaginary part  of  resonances  arising  in  this system  follows a type of  Oppenheimer's law \cite{Op} when the intensity of the field vanishes.


 \section{ Main results}
 
 \subsection{Setting}
 
 Before  describing   the  main results of the paper  we want to  recast the problem into a more convenient form. This allows us to state precisely our assumptions on the system.

 Consider a curved strip $ {\bf \Omega} $ in $\mathbb{R}^{2}$ of a constant width $d$  defined around a smooth reference curve $\Gamma$,  we suppose that  $ {\bf \Omega} $ is not self-intersecting.   The points ${\bf X}=(x,y)$ of $ \bf \Omega$ 	are described by the curvilinear coordinates $(s,u) \in \mathbb R \times (0,d)$, 
\begin{align}
\label{cc}
x&=a(s)-ub^{'}(s),\notag\\
y&=b(s)+ua^{'}(s),
\end{align}
 where $a, b$ are smooth functions  defining  the reference curve $ \Gamma=\{(a(s),b(s)), s\in \mathbb{R}\}$ in $\mathbb{R}^{2}$. They are supposed  to satisfy $a'(s)^{2}+b'(s)^{2}=1$.

Introduce the signed curvature $\gamma(s)$ of $\Gamma$,
\begin{equation}
\label{sc}
\gamma(s)=b'(s)a''(s)-a'(s)b''(s).
\end{equation}
 For a given curvature $\gamma$,   the functions $a$ and $b$ can be  chosen  as 
\begin{align}
\label{courbe}
a(s)&=\int_{0}^{s}\cos\alpha(t)\,dt,\;b(s)=\int_{0}^{s}\sin\alpha(t)\,dt,
\end{align}
where
 $\alpha(s_{1},s_{2})=-\int_{s_{2}}^{s_{1}}\gamma(t)\,dt$ is the angle between the tangent vectors to $\Gamma$ at the points $s_{1}$ and $s_{2}$  (See e.g. \cite{DE} or \cite{Exner} for more details).
 Set $\alpha(s)= \alpha(s,0), s \in \mathbb R$ and  $\alpha_{0}=\alpha (s_{0})$.
 We choose $\gamma$  with  a compact support, ${\rm supp}(\gamma)=[0,s_{0}]$ for some $s_{0}> 0$. In particular for $s<0$ the strip is straight, parallel  to the $x-$axis. Assume also that 
  \begin{itemize}
  \item [(h1)] $\gamma \in \mathrm{C}^{2}(\mathbb R)$,
   \item [(h2)] $d\|\gamma\|_{\infty}< 1$.
   \end{itemize}
   Evidently this implies that $\gamma$ has a   continuous and bounded derivative up to second order.
 
 Let ${\bf F}=F(\cos(\eta), \sin(\eta))$ be  the electric field.   In this work   the intensity of the field $F>0$ is the free parameter   and   the direction $\eta$  is fixed. It satisfies   
 \begin{itemize}
  \item [(h3)] $ |\eta|<\frac{\pi}{2}$  and $|\eta-\alpha_{0}|<\frac{\pi}{2}$.
 \end{itemize}
 See the  Remark  \ref{rm1} below for a discussion about assumptions on $\eta$.  We consider   the Stark effect Hamiltonian on $L^{2}(\bf \Omega)$,
\begin{equation}
\label{ST}
{ \bf H} (F)=-\Delta_{ {\bf \Omega} }+ {\bf F}\cdot {\bf X },\;  F>0,
\end{equation}
with Dirichlet boundary conditions on $ \partial {\bf \Omega} $, the boundary of $ {\bf \Omega} $. 
One can check  that under (h1) and (h2), then by using natural curvilinear coordinates, ${ \bf H} (F)$ is unitarily equivalent to the Schr\"{o}dinger operator defined by

   \begin{equation}
\label{hamilt}
H(F)=H_{0}(F)+ V_{0},\;\;H_{0}(F)= H_0  + W(F), \;\;  H_0 =T_s +T_u 
\end{equation}
on the Hilbert space $L^{2}(\Omega)$,  $\Omega= \mathbb{R}\times (0,d)$  with Dirichlet boundary conditions on $ \partial \Omega = \mathbb R \times \{0, d\}$. Here
\begin{equation}
T_{s} = -\partial_{s}g\partial_{s},\;  g=  g(s,u) =(1+u \gamma(s))^{-2}, \; T_u = -\partial_{u}^{2}
\end{equation}
 and $W(F) $ is the operator multiplication by the function,

\begin{equation}
W(F,s,u) =\left\{
\begin{array}{ll}
F(\cos(\eta) s  + \sin(\eta) u) & \hbox{if $ s < 0 $} \\\\
F (\int_{0}^{s}\cos(\eta-\alpha(t))\,dt+ \sin(\eta-\alpha(s))u)&\hbox{if $0\leq s\leq s_{0}$}\\\\
F(\cos(\eta-\alpha_{0})(s-s_{0}) + A + \sin(\eta-\alpha_{0})u)&\hbox{if $s > s_{0}$}\\\\
\end{array}
\right.
\end{equation}
where $A= \int_{0}^{s_{0}}\cos(\eta-\alpha(t))\,dt$ 

and

\begin{equation}
V_{0}(s,u)=-\frac{\gamma(s)^{2}}{4(1+u\gamma(s))^{2}}+\frac{u\gamma''(s)}{2(1+u\gamma(s))^{3}}-\frac{5}{4}\frac{u^{2}\gamma'(s)^{2}}{(1+u\gamma(s))^{4}}.
\end{equation}

Denote by $ H= H_0 + V_0$.  This is  the hamiltonian   associated with the   guide  in absence of  electric field. If (h1) and (h2) are satisfied, then   $H$ is   a
self-adjoint operator  on $ L^2(\Omega)$ with domain $D(H)$ coinciding  with the one of $H_0$ and  \cite{Ka}
\begin{equation} \label {DH}
 D(H)  = D(H_0)=  \{ \varphi  \in {\cal H}_0^1 ( \Omega), \; H_0\varphi \in L^2(\Omega)\}
 \end{equation}
 In this paper we use standard notation from Sobolev space theory. The essential spectrum of $H$,  $ \sigma_{ess} (H)= [\lambda_0, + \infty)$ where $\lambda_0$ is the first transverse mode of the system i.e
 the first eigenvalue of the operator $ T_u $ on $L^2(0,d)$ with  Dirichlet  boundary conditions at $\{0,d\}$.
 
   Moreover  under our assumptions the operator $H$ has at least one  discrete eigenvalue below the essential spectrum (see  \cite{DE}). Although we do not know the discrete spectrum of $H$
our study below   works  even in the case where $H$ has infinitely many distinct discrete eigenvalues (possibly degenerate)   which can accumulate at
 the threshold $\lambda_0$.

\begin{remark}  \label{rm1} The situation  where $  \eta =  \frac {\pi}{2}$  and $ \alpha_0 = 0$ has been  considered in \cite{Exner}, but in that case  there is no  Stark resonance.   It is also true if we suppose  $ \vert   \eta \vert  \geq   \frac {\pi}{2}$  and    $\vert  \eta  - \alpha_0 \vert< \frac {\pi}{2}$ since    $W(F)$   is   now a confining potential.  Note that    the regime  $ \vert \eta \vert > \frac {\pi}{2} $   and $ \vert  \eta  - \alpha_0 \vert>\frac {\pi}{2}$  is  a symmetric case   of the  one considered  in this paper and can be studied   in the same way.  
While   for $ \vert   \eta \vert  <   \frac {\pi}{2}$ and    $\vert  \eta  - \alpha_0 \vert> \frac {\pi}{2}$,  the situation is  quite different  since   $W(F) \to -\infty$  at both $ s \to \pm \infty$    the "escape" region  corresponding to any negative energy    contains   $  \{ s <-a, u \in (0,d)\}  \cup  \{ s >a, u \in (0,d)\}$ for some $a>0$ and large.  This needs  slight modifications of our method.   It    is actually    studying  in \cite{Ga}.

\end{remark} 


\subsection{Results}

In this section we give the main results of the paper.  Some minor points  will be specified later in the text. 

First we   need to  define  rigorously the  Stark Hamiltonian   associated to our system. 
 Here we adopt a common fact  in the  literature    about Stark operators  i.e.  $H(F)$ is well  defined as an essentially self-adjoint operator  on $L^2(\Omega)$ \cite{CFKS, Ka, reed simon}.
  In the appendix of the paper  where  the proof of the theorem below is proved,  we 
 give a core  for $H(F)$.
  \begin{theorem} \label{2.1}
 Suppose that   (h1)  and (h2)  hold, then for $F>0$,
 
(i) $ {H}(F)$ is  an essentially self-adjoint  operator on $ L^2(\Omega)$,  We will denote the closure of $ {H}(F)$ by the same symbol. 

 (ii)  The spectrum of $H(F)$, $\sigma({H}(F)) = \mathbb R $.
\end{theorem}

  We  now are focusing on  the second main result and its proof.  For any  subset   $\cal D$ of $\mathbb C$, 
  denote by ${\cal D}^- = \{ z \in {\cal D} , \Im z \leq 0\}$.
 
 \begin{theorem}  \label{t0} Suppose  that    (h1), (h2) and (h3)  hold. Let  $E_{0}$ be  an  discrete eigenvalue   of $H $ of  finite multiplicity $j \in \mathbb N$.  There exits $F_c >0 $, a  F-independent complex neighbourhood $\nu_{E_0} $ of  the semi axis  $ (- \infty, E_0 + \frac{1}{8} (\lambda_0- E_0)]$ and  a F-independent  dense  subset  $ \mathcal A $  of $L^2(\Omega) $such that 
 for   $ 0<F \leq  F_c $ 
 
 i) the function 
$$z \in \mathbb C, \Im z >0 \to  {\cal R}_{\varphi}(z)=  \big( ( H(F)-z)^{-1}\varphi, \varphi),\;  \varphi \in \mathcal A  $$
has an meromorphic extension in $\nu_{E_0} $ through  the cut due to  the spectrum of $H(F)$. 
ii) $ \cup_{\varphi \in  \mathcal A}\{   { \rm poles \; of }  \; {\cal R}_\varphi(z) \} \cap \nu_{E_0}^-  $ contains  $j$ poles $Z_0(F),...Z_j(F)$  converging to $E_0$ when $F\to0$.

\end{theorem}
Here    resonances    of the   stark operator  $H(F)$ are defined   as  the set \cite {reed simonIV}
$$ \cup_{\varphi \in  \mathcal A}\{ { \rm  poles \; of }  \; {\cal R}_\varphi(z) \} \cap \mathbb C^- . $$

The resonances have necessarily  a negative imaginary part.  But as it is discussed in the introduction,  a  still open question    is concerned with  the strict negativity.

We shall show below that using the distortion theory, the   resonances coincide with discrete eigenvalues 
of a non self-adjoint operator.


Finally we get the following exponential bound on the width of resonances.
\begin{theorem} \label{pc11}
Under conditions  of the Theorem \ref{t0}.  Let $E_0$  be a simple eigenvalue of $H$ and $Z_0$ the corresponding resonance  for $H(F)$ given by the Theorem \ref{t0}. Then  there exist  two constants   $ 0< c_1,c_2$ such that for $0< F \leq  F_c $,
$$  \vert \Im Z_0 \vert \leq  c_1 e^{-\frac{c_2}{F}}, \; k=1, .. ,j. $$
  \end{theorem}

\begin{remark} Our results exhibit a critical field value $F_c$.  For the Theorem \ref{t0} i) this value is estimated explicitly (see formula \eqref{F0} below).  But this is not true for the rest of the results since our method use     certain abstract analysis  arguments which are valid for $F$ small enough.  This does not give  an explicit  critical $F_c$.

\end{remark}

\section{The distortion theory} \label{The distortion}

  In this section, by using  the distortion theory  we construct a family of non self adjoint  operators  $\{{H}_{\theta}(F), \theta \in \mathbb C,   \vert  \Im \theta \vert  < \theta_{0} \}$ for some $  \theta_0 >0$. In the next section we will see  that under conditions  the discrete spectrum of  ${H}_{\theta}(F)$ coincides with resonances of $H(F)$. We refer the reader to \cite{AC,WH,reed simon}  for basic tools of the    distortion theory.
Here we assume that (h1), (h2) and (h3) are satisfied.  To give  a sense to the construction  below we  need to consider     electric fields   of  finite  magnitude. Without  loss of generality we may suppose  in the sequel  that $0<F \leq1$.

Introduce the distortion on $\Omega$,
\begin{equation}
\label{dsit}
S_{\theta}:(s,u)\mapsto(s + \theta f(s),u)
\end{equation}
 defined from the  vector field $f=-\frac{1}{F\cos(\eta)} \Phi$ where $\Phi\in\mathrm{C}^{\infty}(\mathbb{R})$ is  as follow. Let $E< 0$, be the reference energy, 
 $ 0<\delta E< \frac {1 }{2} \min  \{ 1,\vert E \vert \}   $,  $E_{-}= E - \delta E $ and  $E_{+}=E + \delta E$. Set  $\Phi(s)= \phi [F\cos(\eta) s]$ where $\phi\in\mathrm{C}^{\infty}(\mathbb{R})$ is a non-increasing function such that
\begin{equation}
\phi(t)=1 \,\;\mbox{if}\;\, t< E, \, \, \,\phi(t)=0 \,\;\mbox{if} \,\;t>  E_{+}
\end{equation}
and satisfying $\Vert \phi^{(k)} \Vert_\infty = o((\frac {1}{\delta E})^{k}$.
Note that  for $s < \frac{E}{F\cos(\eta)}$,  $S_\theta$  coincides with  a translation  w.r.t.  the longitudinal variable $s$.

Clearly  for $k \geq 1$,   $\Vert \Phi^{(k)} \Vert_\infty \leq (\frac{F}{\delta E})^{k}$ and   $ \Vert f^{(k)} \Vert_\infty \leq \frac{F^{k-1}}{(\delta E )^{k}}$. 
For $\theta\in\mathbb{R}$, $ \vert \theta \vert < \delta E $, $S_{\theta}$ implements a family of unitary operators on $L^{2}(\Omega)$ by
\begin{equation}
\label{Steta}
U_{\theta}\psi = (1+\theta f')^{\frac{1}{2}}\psi\circ S_{\theta}.
\end{equation}
 We note that
\begin{equation}
\label{unitaire}
 H_{\theta}(F)=U_{\theta}H(F) U_{\theta}^{-1}= H_{0, \theta}(F)+V_{0}.
\end{equation}
\begin{equation} \label{HOtheta}
H_{0, \theta}(F)=T_{s,\theta} + T_{u} + W_{\theta}(F)
\end{equation}
where
\begin{equation}
\label{ts1}
T_{s,\theta}=-(1+\theta f')^{-\frac{1}{2}}\partial_{s}(1+\theta f')^{-1}g\partial_{s}(1+\theta f')^{-\frac{1}{2}},\end{equation}
\begin{equation} \label {Wtheta}
W_{\theta}(F)= W(F)\circ S_{\theta}.
\end{equation}

We  now want to extend  the definition of  ${H}_{\theta}(F)$ for complex $\theta$.
Set $ \theta_0 = \alpha \delta E$  where   $ \alpha $   is a   some  small and strictly positive  constant  which we fix 
 in the proof of the  Theorem \ref{aly ht}  below.   In fact    $\theta_0$ is  the critical value of distortion parameter. 
\begin{proposition}
\label{aly ht}
 There exists $ 0<\alpha<1/2$ independent of $E$ and $F$  such that  for $0< F< \delta E$, 
 $\{{H}_{\theta}(F),  \vert  \Im \theta \vert  < \theta_{0} \}$  is a  self-adjoint analytic family of operators (see \cite{Ka}).
\end{proposition}
\begin{proof}
An computation shows that
\begin{equation} \label{Ts2}
T_{s,\theta}= -\partial_{s}(1+\theta f')^{-2}g\partial_{s} + R_\theta,
\end{equation}
where $R_\theta=   \frac{g}{2}\frac{\theta f^{'''} }{(1+\theta f')^{3}}  - \frac{5g}{4}\frac{\theta^2 f''^{2}}{(1+\theta f')^4} $
is a bounded function. Let $ h(F)=H_0 + w(F)$ be the operator in $L^2(\Omega)$ where $ w(F) $ is the multiplication operator by
\begin{equation} \label{w}
w(F,s) =\left\{
\begin{array}{ll}
F \cos(\eta) s  & \hbox{if $ s < 0 $} \\\\
0 &\hbox{if $0\leq s\leq s_{0}$}\\\\
F \cos(\eta-\alpha_{0})s &\hbox{if $s > s_{0}$.}\\\\
\end{array} \right.
\end{equation} 
 Since $h= h(F) $ differs  from $H(F)$ by adding a bounded  symmetric operator, it   is also  a self-adjoint operator.

 We have  
\begin{equation} \label{E1}
{H}_{\theta}(F)=  h +  \partial_{s}G_{\theta} \partial_{s} +R_\theta +  W_\theta(F) -w(F) + V_0, \; G_{\theta}= g \big (\frac{2\theta f'+\theta^2 f'^{2}}{(1+\theta f')^{2}} \big).
\end{equation}
 Let us show  that for $  \vert  \theta \vert $ small enough then $D({H}_{0,\theta}(F)) = D(h)$. Through unitary  equivalence we may suppose that $\Re \theta =0$.
  In view of the perturbation theory  \cite{Ka} and  \eqref{E1}we only need to to show that $\partial_{s}G_{\theta}\partial_{s}$ is $h$-bounded with a relative bound strictly  smaller than one.
 By using  the  resolvent identity,
\begin{eqnarray}
\partial_{s}G_{\theta}\partial_{s}(h+ i)^{-1}&=&\partial_{s}G_{\theta}\partial_{s}(H_{0}+ i)^{-1} - \partial_{s}G_{\theta}\partial_{s}(H_{0}+ i)^{-1} w(F) (h+ i)^{-1}\notag\\
&=& \partial_{s}G_{\theta}\partial_{s}(H_{0}+ i)^{-1} -  \partial_{s}G_{\theta}\partial_{s} Fs (H_{0}+ i)^{-1} \frac{w(F)}{Fs}(h+i)^{-1}\notag\\
&-&  \partial_{s}G_{\theta}\partial_{s}(H_{0}+ i)^{-1}(\partial_{s}g +g \partial_{s})(H_{0}+ i)^{-1} \frac{w(F)}{s}(h+ i)^{-1}\notag.
\end{eqnarray}

We know that $D(H_{0}) \subset \mathcal H_{\rm loc}^{2}(\bar \Omega) \cap \mathcal H_{0}^{1}( \Omega) $, \cite{Ad, D, Kri}. Let $ \chi$  be a characteristic function of  
${\rm supp}(f') $.
Then by the closed graph theorem  \cite{Ka} $\partial_{s}g(H_{0}+ i)^{-1}$, $ g \partial_{s} (H_{0}+ i)^{-1}$,  $ \chi  \partial_{s}g\partial_{s}(H_{0}+ i)^{-1}$ and  $ \chi\partial_{s}g\partial_{s} Fs (H_{0}+ i)^{-1} $ are bounded operators. The multiplication operators  $\frac{ w(F)}{s}$ and  $\frac{ w(F)}{Fs}$ are  also bounded. Hence  this is  true for  the operator $\partial_{s}G_{\theta}\partial_{s}(h+ i)^{-1}$.

It is easy to check that under  conditions on parameters $\theta$ and  $ F$, 
$$ \Vert \partial_{s}G_{\theta}\partial_{s}(H_{0}+ i)^{-1} \Vert  \leq \frac{3\alpha}{(1-\alpha)^3} \big( \Vert  \chi  \partial_{s}g\partial_{s}(H_{0}+ i)^{-1}\Vert +  \Vert g\partial_{s}(H_{0}+ i)^{-1}\Vert\big) \leq  C\frac{3\alpha}{(1-\alpha)^3} $$
 for some constant $C>0$ independent of $F$ and $E$. Evidently  $ \Vert  \partial_{s}G_{\theta}\partial_{s} Fs (H_{0}+ i)^{-1} \Vert  $  satisfies a   similar estimate.  Choosing  $\alpha$  so  small such
    that  $  \Vert   \partial_{s}G_{\theta}\partial_{s}(h+ i)^{-1} \Vert <1$,   then  $  \partial_{s}G_{\theta}\partial_{s}$ is  relatively bounded to $ h$ with relative bound strictly smaller that one. Thus  the statement follows.

The proof is complete if we can show that for $\psi \in D(H_{\theta}(F))$
\begin{equation*}
\theta \in \{ \theta \in \mathbb C,  | \theta | < \theta_{0} \} \longmapsto (H_\theta(F)\psi, \psi)
\end{equation*}
is an  analytic function. But this last fact can be readily  verified by using standard arguments of measure theory  and  the explicit expression 
  \eqref{E1} (see e.g. \cite{reed simon, Ka}).

\end{proof}

\begin{remark}   
For $\theta \in \mathbb R$, $\vert  \theta \vert < \delta E $ consider   the unitary transformation on $L^2(\mathbb R)$ 
  $$u_{\theta}\psi (s)= (1+\theta f'(s))^{\frac{1}{2}}\psi(s + \theta f(s)),\; \psi \in  L^2(\mathbb R).$$
   We know from \cite{WH, reed simonIV} that  there exists      a dense subset  of analytic vectors $\psi$ associated with $ u_\theta$   in   $\vert  \theta \vert <  \frac{ \delta E }{\sqrt 2}$ i.e. $ \theta \in \mathbb R \to u_{\theta}\psi $ has an  $L^2(\mathbb R)$-analytic extension  in $\vert  \theta \vert <  \frac{ \delta E }{\sqrt 2}$. denote this set as $\mathcal{A}_1$.
  It is  shown in \cite{WH} that $\mathcal{A}_1$ is dense in $L^{2}(\mathbb R)$.
 Let $ \mathcal A  $ be the linear subspace generated by  vectors of the form $ \varphi \otimes \psi,    \varphi \in {\mathcal A}_1, \psi \in  L^2\big( (0,d ) \big)$.
 Then $ \mathcal A$ is a dense subset  of analytic vectors associated to
the transformation $U_\theta$ in $\vert \theta \vert < \theta_0$.\end{remark} 

For further developments we need to introduce the following   modified operator on $L^2(\Omega)$. Let  $s_{1}> s_{0}$, such that $\cos(\eta-\alpha_{0})(s-s_{0}) + A + \sin(\eta-\alpha_{0})u \geq 0$ for all $u \in (0,d)$. Set 
\begin{equation}
\label{htilde}
\widetilde{H}_0(F)=  H_0 + \widetilde{W}(F),
\end{equation}
where $\widetilde{W}(F)$ is a multiplication operator by 
\begin{equation}
\widetilde{W}(F,s,u)=\left\{%
\begin{array}{ll}
W(F,s,u)& \hbox{if $ s < 0 $, $s> s_{1}$} \\\\
0 &\hbox{if $0\leq s\leq s_{1}$}.\\\\
\end{array}%
\right.
\end{equation}

For $ \theta \in \mathbb R$,  $  \vert \theta \vert <  \theta_0 $,  let $\widetilde{H}_{0, \theta}(F) = U_{\theta}\widetilde H_0(F) U_{\theta}^{-1} = T_{s,\theta} + T_{u}+  \widetilde{W}_{\theta}(F) $. Then we have 

\begin{corollary}
  For $0< F< \delta E$, $ \{\tilde H_{0,\theta}(F),\, \vert  \Im \theta \vert <  \theta_0 \} $  is a self-adjoint analytic family of operators.
\end{corollary}
\begin{proof}
We have
$H_{\theta}(F)-\widetilde{H}_{0,\theta}(F)  = V_{0} + W_{\theta}(F)-\widetilde{W}_{\theta}(F)$,
but
 $V_0$ as well as $W_{\theta}(F)-\widetilde{W}_{\theta}(F)$
are  bounded  and $\theta$-independent so   by  the Proposition \ref{aly ht}  the corollary follows.\end{proof}
\section{Meromorphic  extension of the resolvent.} \label{weak}

Let  $\theta= i \beta$,   we suppose that $
0<  \beta <  \theta_0 $.  Set  
 \begin{equation}
\mu_{\theta}= 1+ \theta f^{\sharp}
\end{equation}
with $f^{\sharp}= \Phi-1$  and  $\Phi$  defined in  the Section \ref{The distortion}. Then the multiplier operator by the function  $\mu_{\theta}$  defined a one to one map from 
$D(H_\theta (F))$ to $D(H_\theta (F))$.    Recall that  $\lambda_0=inf \sigma(T_{u})$  is   the first transverse mode. Let  
\begin{equation} \label{nu}
\nu_{\theta}=\{z\in \mathrm{C},\,\Im\mu_{\theta}^{2}(E_{-} + \lambda_{0}-z)< \beta \frac{\delta E}{2}\} .
\end{equation}
$\nu^c_{\theta}$ denotes  its complement in $\mathbb C$. It is easy to  see that 
  $ \nu_{\theta}$   contains a  $F$-independent complex neighbourhood of the semi axis $ (- \infty, \lambda_0 +E - \frac{3}{4} \delta E]$ denoted by $\tilde \nu_\theta$. It is  defined as 
\begin{equation} \label{nut}
\tilde \nu_\theta = \{ x \leq 0, y \geq - \frac{\beta\delta E}{2}\} \cup\{ (x  > 0, y  \geq 2 \beta x - \frac{\beta\delta E}{2}\}
\end{equation}
where $x= \Re z - \lambda_0 -E_-$ and $y=  \Im z$.

In this section our main result is the following. Let 
\begin{equation}
\label{F0} 
F_0 = 
\alpha' (\delta E)^2 \min\{1,1/d\}
\end{equation}
 where $\alpha'$ is a strictly positive  constant  independent of $E$ and $\beta$ which is determined in the proof of the Lemma  \ref{mutetaaaaa}. We have

\begin{proposition}  \label{t1}  There exits $\alpha'  >0 $  such that for all  $E <0$,  $ 0<F \leq  F_0 $, the function 
$$z \in \mathbb C, \Im z >0 \to  {\cal R}_{\varphi}(z)=  \big( ( H(F)-z)^{-1}\varphi, \varphi),\;  \varphi \in \mathcal A  $$
has an meromorphic extension in $ \cup_{0 < \beta < \theta_0}\nu_{\beta}$. 
\end{proposition}

 As a consequence of the Proposition \ref{t1}, the Theorem \ref{t0} i) is proved.
 
The proof of  the Proposition \ref{t1}  is based on the two following results.  For a given operator $O$ on $L^2(\Omega)$ we denote by $\varrho(O)$ its  the resolvent set. 
 \begin{lemma} 
\label{mutetaaaaa}   There exits $ \alpha'>0$   such that for  $ E<0$,  $0< F \leq  F_0$  and $0<  \beta <  \theta_0 $. Then 
\begin{itemize}
  \item[(i)] $\nu_{\theta} \subset \varrho(\widetilde{H}_{0,\theta}(F))$.
  \item [(ii)] $\forall z\in \nu_{\theta}$, $\|(\widetilde{H}_{0,\theta}(F)-z)^{-1}\| \leq {\rm dist}^{-1} (z, \nu^c_{\theta}).$
\end{itemize}
\end{lemma}
\begin{proof}
By using a standard commutation relation  we derive from (\ref{ts1}),
\begin{equation}
\label{fqu}
\mu_{\theta}T_{s,\theta}\mu_{\theta}= T_{1}(\theta) + i T_{2}(\theta) + \mu_{\theta}(T_{s,\theta}\mu_{\theta})
\end{equation}
where $T_{1}(\theta)=-\partial_{s}\Re\{ \mu_{\theta}^{2} (1+ \theta f')^{-2}\}g\partial_{s}$, $T_{2}(\theta)=-\partial_{s}\Im\{ \mu_{\theta}^{2}(1+ \theta f')^{-2}\}g\partial_{s}$.  The operators $T_{1}(\theta)$, $T_{2}(\theta)$ are symmetric  and we know from \cite{PB2}  that $T_{2}(\theta)$ is negative. Moreover a straightforward calculation shows
\begin{equation}
\label{ts}
\Im\mu_{\theta}(T_{s,\theta}\mu_{\theta})= O\bigg ( \frac{\beta F^{2}}{(\delta E)^3}\bigg ).
\end{equation}
In the other hand,  let $z\in \nu_{\theta}$,   set $ \beta S =- \Im\mu_{\theta}^{2}(\widetilde{W}_{\theta}(F)-E_{-} ) -  \Im\mu_{\theta}(T_{s,\theta}\mu_{\theta})$ in fact
\begin{equation}
  S=   (1-\beta^{2}{f^{\sharp}}^{2})\Phi - 2f^{\sharp} \big ( \widetilde{W}(F)-E_{-}  \big) - \beta^{-1}\Im\mu_{\theta}(T_{s,\theta}\mu_{\theta}).
\end{equation}
On  ${\rm supp}(f^{\sharp}) = { \rm supp}(\Phi-1)$, we have $\cos(\eta-\alpha_{0})(s-s_{0})+ \sin(\eta-\alpha_{0})u + A \geq 0$  if $ s > s_1$, $ F
 \cos(\eta)s -E_{-}\geq \delta E$ if $s<0$ and then 
$$
F\cos(\eta) s \chi_{\{s<0\}}+ F(\cos(\eta-\alpha_{0})+ \sin(\eta-\alpha_{0})u + A)\chi_{\{s\geq s_{1}\}} - E_{-} \geq  
 \delta E \chi_{\{s<0\}} - E_{-}\chi_{\{s\geq 0\}}\geq  \delta E. \notag
$$
By using  \eqref{ts},  we get for  $0< \beta < \theta_0 $ 
$$S \geq 
\frac {1} { 2}\Phi + 2(1-\Phi)(\delta E + F u \sin\eta\chi_{\{s<0\}}) + O \bigg( \frac{F^{2}}{(\delta E)^3} \bigg).$$
Then    we can choose  $\alpha'$  so small   such that, 
$$ S \geq \frac {1} { 2}   \min \{ \frac{1}{2}, \delta E \}= \frac {\delta E} { 2}. $$ 

Further  in the quadratic form sense  on $ D({H}_{\theta}(F)) \times D({H}_{\theta}(F))$, we have 
\begin{equation}
Im\mu_{\theta}(\widetilde{H}_{0,\theta}(F)-z)\mu_{\theta}= T_2(\theta) -\beta S  + \Im\mu_{\theta}^{2} T_u+Im\mu_{\theta}^{2}(E_{-}-z).  
\end{equation}
Thus  for $0<   \beta  < \theta_0 $, $0< F  \leq    F_0 $ and    $z\in \nu_{\theta}$, since $\Im\mu_{\theta}^{2} = 2 \beta f^{\sharp}\leq 0$, 
we get
\begin{eqnarray}
Im\mu_{\theta}(\widetilde{H}_{0,\theta}(F)-z)\mu_{\theta}&\leq& -\beta \frac{\delta E}{2}+ \Im\mu_{\theta}^{2}(E_{-}+\lambda_{0}-z)<0.
\end{eqnarray}
This last estimate with together  some usual arguments for non-trapping estimates given in \cite{PB2} complete the proof of the Lemma (\ref{mutetaaaaa}).
\end{proof}


Introduce the following  operator, let  $\theta  \in \mathbb C$, $\vert \theta \vert < \theta_0 $ and $z\in\nu_{\theta}$
\begin{equation}
\label{kteta}
K_{\theta}(F,z)=(V_{0}+ W_{\theta}(F)-\widetilde{W}_{\theta}(F))(\widetilde{H}_{0,\theta}(F)-z)^{-1}.
\end{equation}

\begin{lemma} \label{l2}
\begin{itemize} In the same conditions as in  the previous lemma.
\item [(i) ] $z \in \nu_\theta \to K_{\theta}(F,z)$ is  an analytic compact  operator valued function.
\item [(ii) ] For $z\in  \nu_\theta$, $Imz> 0$ large enough, $\|K_{\theta}(F,z)\| < 1$.
\end{itemize}
\end{lemma}
\begin{proof}
 By the  Lemma  \ref{mutetaaaaa}, $(i)$  follows if we show   that $K_{\theta}(F,z), z \in \nu_\theta $ are compact operators.  Set $V=V_{0}+ W_{\theta}(F)-\widetilde{W}_{\theta}(F)$.  Notice that $V$ has compact support in the longitudinal direction and it is a bounded operator.

 Introduce  the operator $  \tilde h=\tilde h(F) = H_0 + \tilde w(F)$ on $L^2(\Omega)$
where $\tilde w(F)$   is the multiplication operator by 
\begin{equation} \label{tildew}
 \tilde w(F,s) =\left\{
\begin{array}{ll}
F \cos(\eta) s  & \hbox{if $ s < 0 $} \\\\
0 &\hbox{if $0\leq s\leq s_{1}$}\\\\
F \cos(\eta-\alpha_{0})s &\hbox{if $s > s_{1}$.}\\\\
\end{array} \right.
\end{equation} 
Then 
\begin{equation} \label{E2}
\widetilde{H}_{0,\theta}(F)-  \tilde h=    \partial_{s}G_{\theta} \partial_{s} +R_\theta + \tilde W_\theta(F) -\tilde w(F), 
\end{equation}
where $R_\theta$, $G_{\theta}$ and $\widetilde{W}_{\theta}(F) $ are  defined in  the Section 3. Suppose $ \vert \theta \vert < \theta_0$, $0< F< \delta E$,  this is satisfied  under assumptions of the lemma.   Then following step by step  the proof of the Proposition\ref{aly ht},    $\widetilde{H}_{0,\theta}(F)-  \tilde h$ is $\tilde h$-bounded with  a relative bound smaller than one. Therefore, to  prove $(i) $  we are left to  show that   for  $z\in  \nu_\theta, \Im z \not=0$, $V( \tilde h-z )^{-1}$  is compact.  

Denote  by  $\mathbb I_{\cal H}$ the identity operator on  the space ${\cal H}$. Let $h_0= -\partial_s^2  \otimes \mathbb I_{L^2(0,d)} +   \mathbb I_{L^2(\mathbb R)} \otimes T_u $  and $G= g-1$, we have 
\begin{equation} \label{e1}
V(\tilde h-z)^{-1}= V (h_{0}-z)^{-1}+ V(h_{0}-z)^{-1}\big(\partial_s G \partial_s- \tilde  w(F) \big)(\tilde h-z )^{-1} \end{equation}
 Note   that by using again the Herbst's argument \cite{He}, the second term of the r.h.s of \eqref{e1}
 can be written as 
 \begin{multline} \label{e11}
  V(h_{0}-z)^{-1} \tilde  w(F)(\tilde h-z )^{-1} =V s (h_{0}-z)^{-1} \frac{\tilde w(F)}{s}(\tilde h-z )^{-1} + \\
 V(h_{0}-z)^{-1}[s,h_{0}](h_{0}-z)^{-1}\frac{\tilde w(F)}{s} (\tilde h-z)^{-1}.
\end{multline}
  In the one hand  let $\chi$  be a $ C^\infty $  characteristic function    of  $[0,s_1]$
    then $\chi(h_{0}-z)^{-1}$ is  a compact  operator. Indeed, 
  \begin{equation*}
 \chi (h_{0}-z)^{-1}=\sum_{n\geq 0} \chi (-\partial_s^2+\lambda_{n}-z)^{-1}\otimes p_{n}
 \end{equation*}
 where $\lambda_n, n \in \mathbb N$ are the eigenvalues of the operator $T_u$ (transverse modes) and 
 $p_n, n \in  \mathbb N$ the associated projectors. We know that $\chi (-\partial_s^2+ \lambda_{n}-z)^{-1}\otimes p_{n}$ is compact \cite{reed simon} and 
for large $n$,
\begin{equation} \label{Vcpc}
\|\chi (-\partial_s^2+\lambda_{n}-z)^{-1}\otimes p_{n}\|  \leq  \|(-\partial_s^2+\lambda_{n}-z)^{-1} \Vert
= O (\frac{1}{n^{2}}).
\end{equation}
Thus  $\chi (h_{0}-z)^{-1}$ is   compact   since it is  a    limit of  a sequence of compact
operators in the norm  topology. This  holds  true for operators $V  (h_{0}-z)^{-1}$  and  $V s (h_{0}-z)^{-1}$.


On the other hand   the function  $G$ has a bounded support in the longitudinal direction then the same arguments as in the proof of the Proposition \ref{aly ht} imply that the operator $\partial_s G \partial_s (\tilde h-z )^{-1} $ is bounded.  By the closed graph theorem 
$[s, h_{0}](h_{0}-z)^{-1} = 2 \partial_s  (h_{0}-z)^{-1}$ is also bounded. 

Then   by \eqref{e1} and  \eqref{e11} the statement follows.

  The assertion  $(ii)$   is a direct consequence of the Lemma \ref{mutetaaaaa} $(ii) $ and the fact that $V$ is a bounded operator. 
  \end{proof}
\subsection{Proof of the Proposition \ref{t1}} \label{4.1}

  Here we refer  e.g. to \cite{ reed simonIV}  for the reader unfamiliar with the distortion theory.

 Let $E<0$,  $ \vert \theta \vert < \theta_0$ and $0<F\leq F_0$.  By  Lemmas \ref{mutetaaaaa},  \ref{l2} and the standard  Fredholm alternative  theorem,   the operator $ \mathbb I_{L^2(\Omega)}+  K_{\theta}(F,z)$  is invertible for all    $z \in \nu_\theta \setminus \mathcal R$ where $\mathcal R$ is a discrete  set. In 
 the bounded operator sense, we have  
\begin{equation} \label{Res}
( H_\theta(F)-z)^{-1}= ( \tilde H_{0,\theta}(F)-z)^{-1} \big(\mathbb I_{L^2(\Omega)}+  K_{\theta}(F,z) \big)^{-1}.
\end{equation}
This implies that $ \nu_\theta \setminus \mathcal R \subset \rho( H_\theta(F))$.

Further   let    $ \mathcal O $ an open subset of  $\nu_\theta \setminus \mathcal R $.  For  $ \varphi  \in \mathcal A$, consider 
 the function 
  \begin{equation} \label{fz}
 z \in \mathcal O \to {\cal R}_\varphi(z)= \big(  ( H(F)-z)^{-1}\varphi ,\varphi\big).
 \end{equation}
 For $\theta \in \mathbb R$, $ \vert \theta \vert < \theta_0$, by using  the identity $U_{\theta}^{\ast}U_{\theta}= \mathbb I_{L^2(\Omega)}$ in the scalar product  of the r.h.s. of  \eqref{fz}, we have $ {\cal R}_\varphi(z)=   \big(  ( H_\theta(F)-z)^{-1}\varphi_\theta,\varphi_\theta \big)$,  $\varphi_\theta = U_{\theta}\varphi$.
 Then  together with the  Proposition \ref{aly ht},  it   holds 
 \begin{equation} \label{Rzz}
{\cal R}_\varphi(z)=   \big(  ( H_\theta(F)-z)^{-1}\varphi_\theta,\varphi_{\bar \theta} \big).
\end{equation}
in the  disk $ \{\theta \in \mathbb C,  \vert \theta \vert < \theta_0 \}$.

Fix  $ \theta= i \beta,  0< \vert \beta \vert < \theta_0$ then   ${\cal R}_\varphi$ has  an  meromorphic  extension in 
$\nu_\theta$  given by 
$$ {\cal R}_\varphi(z)=  \big(  ( \tilde H_{0,\theta}(F)-z)^{-1} \big( \mathbb I_{L^2(\Omega)} +  K_{\theta}(F,z) \big)^{-1}\varphi_\theta,\varphi_{\bar \theta} \big).$$
The poles of ${\cal R}_\varphi$ are  locally $\theta $-independent.   From \cite{WH}  and  standard arguments, these poles   are  the set  of $ z \in \mathbb \nu_\theta $ such that the equation
$ K_{\theta}(F,z) \psi = - \psi$ has non-zero solution in $L^2(\Omega)$. In view of \eqref{Res} they  are the discrete eigenvalues of the operator $H_\theta(F)$.
\qed

\section{Resonances.}

This section is devoted to the  proof  ii) of the Theorem \ref{t0}. In view of  the section \ref{4.1} It is given by the following 

\begin{proposition}
\label{pc}
Let  $E_{0}$ be  an  discrete eigenvalue   of $H $ of finite  multiplicity $j \in \mathbb N$.  There exists   $0< F'_0 \leq  F_0 $ such that  for  $0< F \leq  F'_0 $, the operator 
  $H_{\theta}(F)$, $ 0< \vert \theta \vert  < \theta_0 $ has  j eigenvalues near $E_0$  converging  to $E_{0}$ as $F\rightarrow 0$.
\end{proposition}
We need first to show the  following result. For $ \Im z \not= 0$ let $ K(z) = V_0(H_{0}-z)^{-1}$  They are compact operators (see e.g. arguments developed in the Section 6). Note that formally $  K(z) =K_\theta(F=0,z)$.  We have 
\begin{lemma} \label{k2} Let  $ E<0$, $ \theta= i\beta$, $ 0< \beta < \theta_0 $. 
Let $ \kappa$ be a compact subset of  $\tilde \nu_{\theta}\cap \rho(H_{0})$, $\chi= \chi(s) \in C_{0}^{\infty}(\mathbb{R}^+)$. Then
\begin{itemize}
\item [(i) ] $\lim_{F\rightarrow 0} \|(\widetilde{H}_{0,\theta}(F)-z)^{-1} \psi- (H_{0}-z)^{-1}\psi\|=0, \;  \psi  \in L^2(\Omega)$,
\item [(ii) ] $\lim_{F\rightarrow 0} \|\chi(\widetilde{H}_{0,\theta}(F)-z)^{-1}- \chi(H_{0}-z)^{-1}\|=0$,
\item [(iii) ]$\lim_{F\rightarrow 0}  \Vert K_{\theta}(F,z)- K(z) \Vert =0$,
\end{itemize}
uniformly in $z \in \kappa$.
\end{lemma}
\begin{proof}
By using   the arguments of the appendix the operator $H_{0}= T_{s}+T_{u}$ on $L^{2}(\Omega)$  has a core given by  \eqref{Co} i.e.   for $z\in \rho(H_{0})$, $\mathcal C^{'}= (H_{0}-z)\mathcal C$ is dense in $L^{2}( \Omega)$.
 Let $0 < F\leq F_0$ and $z \in \kappa $.  For all $\varphi\in \mathcal C$, set $\psi = (H_{0}-z)\varphi$.  The resolvent equation implies, 
\begin{equation}
(\widetilde{H}_{0,\theta}(F)-z)^{-1}\psi - (H_{0}-z)^{-1}\psi = (\widetilde{H}_{0,\theta}(F)-z)^{-1}( T_{s}- T_{s,\theta} - \widetilde{W}_{\theta}(F))\varphi.
\end{equation}

Clearly     $\lim_{F\rightarrow 0} \|\widetilde{W}_{\theta}(F)\varphi \Vert=0$. On the other hand we have 
$$\Vert (T_{s}- T_{s,\theta})\varphi \Vert  \leq \Vert \partial_s G_\theta \partial_s\varphi  \Vert  + \Vert R_\theta \varphi \Vert $$
Where $G_\theta $ and $R_\theta$ are defined as in the Section 3. 
Evidently $\lim_{F\rightarrow 0} \Vert R_\theta \varphi \Vert =0$. Since ${ \rm supp}(G_\theta)= [\frac{E}{F\cos(\eta)}, \frac{E_+}{F\cos(\eta)} ]$ then for such a $\varphi$,   $\lim_{F\rightarrow 0}\Vert  \partial_s G_\theta \partial_s\varphi  \|=0 $. So that $\lim_{F\rightarrow 0} \|(T_{s}- T_{s,\theta})\varphi \|=0 $.\\
  In view of the Lemma \ref{mutetaaaaa}, $(\widetilde{H}_{0,\theta}(F)-z)^{-1}$ has a norm which   is uniformly bounded  w.r.t. $F$. Thus $ (i)$  is proved on  $\mathcal C^{'}$, by standard arguments then the strong convergence follows.\\
Let us show $(ii)$.  For  $z\in \kappa$ then 
\begin{equation} \label{d1}
\chi(\widetilde{H}_{0,\theta}(F)-z)^{-1}-\chi(H_{0}-z)^{-1}=\chi(\widetilde{H}_{0,\theta}(F)-z)^{-1}Q_{\theta}(F)\end{equation}
where $Q_{\theta}(F)=( T_{s}- T_{s,\theta} - \widetilde{W}_{\theta}(F))(H_{0}-z)^{-1}$.
  On  ${\rm supp}(\chi) $, $f=0$ then the  following resolvent identity holds,
\begin{equation}\label{d22222}
 \chi (\widetilde{H}_{0,\theta}(F)-z)^{-1}= (H_{0}-z)^{-1}\chi + (H_{0}-z)^{-1}([T_s,\chi] - \chi \widetilde{W}(F))(\widetilde{H}_{0,\theta}(F)-z)^{-1}.
\end{equation}
In view of \eqref{d1} and \eqref{d22222} we have to consider two terms. First 
$$ t_1 (F)= (H_{0}-z)^{-1}\chi Q_{\theta}(F)= (H_{0}-z)^{-1}\chi\widetilde{W}(F))(H_{0}-z)^{-1}$$
which  clearly converges in the norm sense to $0_{\mathcal B(L^2(\Omega))}$ as $F\to 0$ uniformly in $z \in \kappa$ and 
$$t_2(F)= (H_{0}-z)^{-1}([T_s,\chi] - \chi \widetilde{W}(F))(\widetilde{H}_{0,\theta}(F)-z)^{-1}Q_{\theta}(F).$$
Let $\bar \chi $ be the characteristic function of  $ {\rm supp}( \chi)$.
We know that the operator $ (H_{0}-z)^{-1} \bar \chi $ is compact (see e.g. the proof of the Lemma \ref{l2}) then 
to prove that $t_2(F) $ converges in the norm sense 
to $0_{\mathcal B(L^2(\Omega))}$ as $F\to 0$ uniformly in $z \in \kappa$,  it is sufficient to 
show that $([T_s,\chi] - \chi \widetilde{W}(F))(\widetilde{H}_{0,\theta}(F)-z)^{-1}Q_{\theta}(F)$ converges strongly to 
$0_{\mathcal B(L^2(\Omega))}$ as $F\to 0$ uniformly in $z \in \kappa$. But considering the proof of $(i) $
it is then sufficient to prove that the operator $([T_s,\chi] - \chi \widetilde{W}(F))(\widetilde{H}_{0,\theta}(F)-z)^{-1}$
is bounded operator and has a norm which   is uniformly bounded  w.r.t. $F$ if $F$ is small  and $z \in \kappa$.\\
Evidently  by the Lemma \ref{mutetaaaaa} this is  true for the operator $\chi \widetilde{W}(F)(\widetilde{H}_{0,\theta}(F)-z)^{-1}$. \\
We have  on $ L^2(\Omega)$, 
$$ [T_s,\chi] (\widetilde{H}_{0,\theta}(F)-z)^{-1}=
-(\chi' g\partial_s + \partial_sg\chi')(\widetilde{H}_{0,\theta}(F)-z)^{-1} = -(2\chi' g\partial_s + (g\chi')')(\widetilde{H}_{0,\theta}(F)-z)^{-1}.$$ 
 Since  the functions $g$ and $  (g\chi')'$ are bounded  and do not dependent  on $F$,   we  only have to consider the operator $\chi' g^{1/2}\partial_s (\widetilde{H}_{0,\theta}(F)-z)^{-1}$.\\
  Let $\varphi\in L^{2}(\Omega)$, $\Vert  \varphi \Vert =1$ set $ \psi = (\widetilde{H}_{0,\theta}(F)-z)^{-1}\varphi$. Integrating by part, we have
\begin{equation*}
\| \chi' g^{1/2}\partial_{s}\psi \|^2 = (-\partial_{s}(\chi' )^2g\partial_{s}\psi,\psi) \leq  (-\partial_{s}(\chi')^2g\partial_{s}\psi,\psi) + (\chi'T_u \chi'\psi,\psi).
\end{equation*}
 By using standard commutation relations,   $ \partial_{s}(\chi' )^2g\partial_{s} = \frac{1}{2} ( (\chi')^2 \partial_{s}g\partial_{s} + \partial_{s}g\partial_{s} (\chi')^2 +  \partial_{s}(g \partial_{s}( \chi')^2))$.
 Since the field $f=0$ on ${\rm supp}( \chi')$ we get,
\begin{eqnarray}  \nonumber 
 \| \chi'g^{1/2}\partial_{s}\psi \|^2 &\leq  &
\Re((\widetilde{H}_{0,\theta}(F)-z)\psi, (\chi')^2\psi) - \Re ((\widetilde{W}_{\theta}(F)-z)\psi,(\chi')^2\psi) + \frac{1}{2}(\partial_{s}(g \partial_{s}( \chi' )^2)\psi, \psi)\notag  \\
 & \leq  & \Vert ( \chi')^2\Vert_\infty  \Vert  (\widetilde{H}_{0,\theta}(F)-z)^{-1}\Vert  +
  \big( \Vert \partial_{s}(g \partial_{s}( \chi')^2)\Vert _\infty  + \\ && \Vert( \chi')^2(\widetilde{W}_{\theta}(F)-z)  \Vert_\infty \big  \Vert  (\widetilde{H}_{0,\theta}(F)-z)^{-1}\Vert^2.
\nonumber \end{eqnarray}
 The Lemma \ref{mutetaaaaa} implies that  the l.h.s. of the last inequality is bounded
 uniformly  w.r.t. $F$ if $F$ is small  and $z \in \kappa$.  
 
 Note that once the strong convergence  on $\mathcal C'$ is proved, the strong convergence on $L^2(\Omega)$ follows by using the fact that
 \begin{eqnarray}  \nonumber ([T_s,\chi] - \chi \widetilde{W}(F))(\widetilde{H}_{0,\theta}(F)-z)^{-1}Q_{\theta}(F) = \\
 ([T_s,\chi] - \chi \widetilde{W}(F)) \big( (\widetilde{H}_{0,\theta}(F)-z)^{-1} -  (H_0 -z)^{-1} \big)
 \nonumber \end{eqnarray}
 is uniformly bounded w.r.t $F$ for $F$ small and $z \in \kappa$.
Hence the proof  of $(ii)$ is done.
 
We have
\begin{eqnarray}
K_{\theta}(F,z)-K(z)&=& (V_{0}+W_\theta(F)-\widetilde{W}_\theta (F))(\widetilde{H}_{0,\theta}(F)-z)^{-1}-V_{0}(H_{0}-z)^{-1}\notag\\
&=& V_{0}((\widetilde{H}_{0,\theta}(F)-z)^{-1}-(H_{0}-z)^{-1})-(W_\theta(F)-\widetilde{W}_\theta (F))(\widetilde{H}_{0,\theta}(F)-z)^{-1}.\notag
\end{eqnarray}
Clearly   in the norm sense $(W_\theta(F)-\widetilde{W}_\theta (F))(\widetilde{H}_{0,\theta}(F)-z)^{-1}\rightarrow  0_{\mathcal B(L^2(\Omega))}$ as $F\rightarrow 0$, uniformly  w.r.t.  $z \in \kappa$.
By applying $(ii)$   this is  also true  for  $V_{0}((\widetilde{H}_{0,\theta}(F)-z)^{-1}-(H_{0}-z)^{-1})$  as $F \to 0$.  Then 
\[\lim_{F\rightarrow 0}\|K_{\theta}(F,z)-K(z)\|=0.\]
uniformly  w.r.t.  $z \in \kappa$.
\end{proof}
\subsection {Proof of the Proposition \ref{pc}}

 Let $E_0 $ be an eigenvalue of   the operator $H$.   Recall that  $\lambda_0 = \inf\sigma_{ess}(H)$. Choose the reference energy, $E$ so that    $E_-=  E_0-\lambda_0 = E - \delta E$ and $ \delta E = \frac { \vert  E \vert}{2}$.

Let $ 0< \vert \theta  \vert  < \theta_0, \Im \theta = \beta >0$. Suppose $R>0$ is such that the  complex disk, $\mathcal{D}=\{ z \in \mathbb C, |z-E_{0}|\leq R\}   \subset \tilde \nu_\theta$ and $\mathcal{D}\cap \sigma (H)=\{E_{0}\}$.\\
First, we show that  for  $F$ small enough,  $z\in\partial\mathcal{D}$,    $(H_{\theta}(F)-z)^{-1} $ exists. Clearly $ H $ has no spectrum  in $\partial\mathcal{D}$  then  in view of the identity 
$$ (H-z)^{-1} = (H_0-z)^{-1} (\mathbb I_{L^2(\Omega)} +K(z))^{-1}, \; z \in \rho(H)\cap\rho(H_0),$$ 
the operator $(\mathbb I_{L^2(\Omega)} +K(z))^{-1}$ is well defined on $\partial\mathcal{D}$ and its norm is uniformly bounded w.r.t. $z \in  \partial\mathcal{D}$.

We have
\begin{equation} \label{1+k}
\mathbb I_{L^2(\Omega)}+ K_{\theta}(F,z)=\bigg(\mathbb I_{L^2(\Omega)} +(K_{\theta}(F,z)-K(z))(\mathbb I_{L^2(\Omega)} +K(z))^{-1}\bigg)(\mathbb I_{L^2(\Omega)} +K(z)).
\end{equation}
Since by the Lemma \ref {k2} $(iii)$, $\|K_{\theta}(F,z)-K(z)\|\rightarrow 0$ as $F\rightarrow 0$ uniformly for $ z\in\partial\mathcal{D}$,  then for $F$ small enough and $z \in \partial\mathcal{D}$
$$\|(\mathbb I_{L^2(\Omega)} +K(z))^{-1}(K_{\theta}(F,z)-K(z))\| < 1. $$
 and   $\mathbb I_{L^2(\Omega)} + (K_{\theta}(F,z)-K(z))(\mathbb I_{L^2(\Omega)}+K(z))^{-1}$ is invertible.
Hence for $F$ small enough  $\mathbb I_{L^2(\Omega)} + K_{\theta}(F,z)$ is invertible for $z\in\partial\mathcal{D}$ and  from \eqref{Res}, $(H_{\theta}(F)-z)^{-1}$ is well defined on the contour $\partial \mathcal{D}$.  We define  the spectral projector associated with $H_\theta (F)$,
\begin{equation} \label{Pt}
P_{\theta}(F)=\frac{1}{2i\pi}\oint_{\partial\mathcal{D}}(H_{\theta}(F)-z)^{-1}\,dz.
\end{equation}
The algebraic multiplicity of the eigenvalues of $H_{\theta}(F)$ inside $\mathcal{D}$ is just the dimension of $P_{\theta}(F)$.
In the same way let 
\begin{equation*}
P=\frac{1}{2i\pi}\oint_{\partial\mathcal{D}}(H-z)^{-1}\,dz
\end{equation*}
 be the spectral projector associated with $H$. Thus  to prove the  first part of the proposition, it is sufficient to  show that for $F$ small enough, $\|P_{\theta}(F)-P\|<1$. We have 
\begin{eqnarray}  \label{DR}
(H_{\theta}(F)-z)^{-1}&=& (\widetilde{H}_{0,\theta}(F)-z)^{-1}(\mathbb I_{L^2(\Omega)} +K_{\theta}(F,z))^{-1}   \\
&=&(\widetilde{H}_{0,\theta}(F)-z)^{-1}-(\widetilde{H}_{0,\theta}(F)-z)^{-1}K_{\theta}(F,z)(\mathbb I_{L^2(\Omega)} +K_{\theta}(F,z))^{-1}  \nonumber
\end{eqnarray}
 and similarly
  \begin{equation*}
(H-z)^{-1}=(H_{0}-z)^{-1}-(H_{0}-z)^{-1}K(z)(\mathbb I_{L^2(\Omega)} +K(z))^{-1}.
\end{equation*}
By the Lemma \ref{mutetaaaaa} the operator $\widetilde{H}_{0,\theta}(F)$  has no spectrum inside $\mathcal D$  this is also true for $H_0$  then  $\oint_{\partial\mathcal{D}}(H_{0}-z)^{-1}\,dz=\oint_{\partial\mathcal{D}}(\tilde H_{0,\theta}(F)-z)^{-1}\,dz=0$.
Hence, we get 
\begin{eqnarray}
P_{\theta}(F)-P=&\frac{1}{2i\pi} \oint_{\partial\mathcal{D}}((H_{0}-z)^{-1}K(z)(\mathbb I_{L^2(\Omega)} +K(z))^{-1} -  \label{p1}\\
 &\widetilde{H}_{0,\theta}(F)-z)^{-1}K_{\theta}(F,z)(\mathbb I +K_{\theta}(F,z))^{-1} \,dz. \nonumber
\end{eqnarray}
Set  $ \Delta K =  K(z)-K_\theta(F,z) $, $ \Delta R=  (H_{0}-z)^{-1}-(\widetilde{H}_{0,\theta}(F)-z)^{-1}$, we have the following identity,
\begin{eqnarray*}
  (H_{0}-z)^{-1}K(z)(\mathbb I_{L^2(\Omega)} +K(z))^{-1} -(\widetilde{H}_{0,\theta}(F)-z)^{-1}K_{\theta}(F,z)(\mathbb I_{L^2(\Omega)} +K_{\theta}(F,z))^{-1}&= \\
 \Delta R K(z)(\mathbb I_{L^2(\Omega)} +K(z))^{-1} + 
 (\widetilde{H}_{0,\theta}(F)-z)^{-1}(\mathbb I_{L^2(\Omega)} +K_{\theta}(F,z))^{-1}\Delta K (\mathbb I_{L^2(\Omega)}+K(z))^{-1}.
\end{eqnarray*}
By applying the Lemma  \ref {k2} then  in the norm operator sense $  \Delta R K(z){\rightarrow}
0_{{\mathcal B } (L^2(\Omega))}$ and  $\Delta K  {\rightarrow}
0_{{\mathcal B } (L^2(\Omega))}$
as $F\rightarrow 0$ uniformly in $z \in \partial\mathcal{D}$. Moreover  the operators $(\mathbb I_{L^2(\Omega)} +K(z))^{-1},
(\mathbb I_{L^2(\Omega)} +K_{\theta}(F,z))^{-1}$ and $(\widetilde{H}_{0,\theta}(F)-z)^{-1}$ are uniformly bounded w.r.t.
 $z \in \partial\mathcal{D}$ and  $F$  for $F$ small. This implies
  \begin{equation} \label{dP}
   \lim_{F\to 0} \|P_{\theta}(F)-P\| =0.
    \end{equation}
The second part of the proposition follows from the fact that   the radius of $\mathcal D$ can be chosen  arbitrarily small, this shows that  the eigenvalues of 
$H_{\theta}(F)$  inside $\mathcal D$ converge to $E_0$  as $F \to 0$. 

\qed

\section{Exponential estimates}

In this section we show that the width of resonances  given in the Proposition
\ref {pc}   decays exponentially  when the intensity of the field $F\to 0$. Hence we prove the Theorem \ref{pc11}.

Let  $E_{0}$ be  an  simple eigenvalue   of $H $.  For $0< F \leq  F'_0 $,  let $Z_0$ be an eigenvalue of  the operator $H_{\theta}(F)$ in a small complex  neighborhood of $E_{0}$
given by the Proposition \ref{pc}. Then 

\begin{proposition} \label{pc1}
Under conditions  of the Theorem \ref{pc},  there exists $ 0 < F''_0 \leq F'_0 $ and two constants   $ 0< c_1,c_2$ such that for $0< F \leq  F''_0 $,
$$  \vert \Im Z_0 \vert \leq  c_1 e^{-\frac{c_2}{F}} $$
  \end{proposition}


 First we need to prove the following lemma.

\begin{lemma} \label{CT} Let $\varphi_0$ be the  eigenvector of $H$ associated with the eigenvalue $E_0$ i.e. $H \varphi_0 = E_0 \varphi_0$. Then there exist $a>0$ 
such that $ e^{a\vert s \vert} \varphi_0 \in L^2( \Omega)$.
\end{lemma}

\proof
Here we use  the standard Combes-Thomas argument (see e.g. \cite{ reed simonIV}). 
Consider the following   unitary  transformation on $L^2(\Omega)$. Let $a \in \mathbb R $, for all $\varphi \in L^2(\Omega)$, set 
$$ W_a(\varphi)(s,u)  = e^{-ias}\varphi(s,u).$$
We have 
$$H_a = W_a HW_a^{-1 } = H -ia( \partial_sg + g \partial_s) + ga^2. $$
The   family  of operators $\{H_a, a \in \mathbb C\} $ is an  entire family  of type A. Indeed it is easy to check that  $ D(H_a) = D(H)$, $\forall a \in \mathbb C$. This follows from  the fact that $ \forall z \in \mathbb C, \Im z \not= 0$,
$$  \Vert g^{1/2} \partial_s ( H-z)^{-1} \Vert  \leq    \Vert ( H-z)^{-1} \Vert  + ( \Vert V_0\Vert_\infty + \vert z \vert )\Vert ( H-z)^{-1} \Vert^2.  $$
Thus, for  a suitable choice of $z$,  the r.h.s of this last inequality is  arbitrarily   small. This implies $\partial_sg + g \partial_s$  is $H$-bounded with  zero relative bound. 

Further  let  $\Re a=0$.    Denote by $ H_{0,a} = H_0 -ia( \partial_sg + g \partial_s) + ga^2$. For    $\varphi \in  D(H) $, we have   
\begin{equation} \label{IN} \Re ( H_{0,a} \varphi, \varphi) = H_{0}  
- g( \Im a)^2 \geq \lambda_0  - g_\infty( \Im a)^2; \;  g_\infty= \Vert g \Vert_\infty.
\end{equation}
Then  for $ z \notin  \Sigma_a= \{ z \in \mathbb C, \Re z \geq \lambda_0  - g_\infty ( \Im a)^2 \}$, $\Vert( H_{0,a}- z)^{-1} \Vert \leq {\rm dist}^{-1 }(z,  \Sigma_a)$ \cite{Ka}.
 Thus if we show that $ V_0 ( H_{0,a}- z)^{-1}$ is compact, then by using usual arguments of the perturbation theory (see e.g; the proof of the Proposition \ref{t1}) 
 the operator $H_a$ has only discrete spectrum in $ \mathbb C \setminus  \Sigma_a$  this will imply  that  the essential spectrum 
of $ H_a$, $\sigma_{ess} (H_a) \subset  \Sigma_a$.

Let  $h_0= -\partial_s^2  \otimes \mathbb I_{L^2(0,d)}+   \mathbb I_{L^2(\mathbb R)} \otimes T_u $  be the operator introduced in the proof of the Lemma \ref{l2} and  $ G= g-1$ we have
$$V_0( H_{0,a}- z)^{-1}= V_0(h_0 -z)^{-1} - V_0 (h_0 -z)^{-1} \big( \partial_s  G\partial_s +ia( \partial_sg + g \partial_s) - ga^2 \big)( H_{0,a}- z)^{-1}.$$ 
 We know that $V_0(h_0 -z)^{-1}$ is compact (see the proof of the Lemma \ref{l2}), so we are left to show that $\big (  \partial_s  G\partial_s +ia( \partial_sg + g \partial_s) - ga^2 \big)( H_{0,a}- z)^{-1}$ is a bounded operator.
We have
$$ \partial_s  G\partial_s (H_{0,a}- z)^{-1}  = \partial_s  G\partial_s(H_0 -z)^{-1}  + \partial_s  G\partial_s(H_0 -z)^{-1} (ia( \partial_sg + g \partial_s) - ga^2) (H_{0,a}- z)^{-1}. $$
since  $D(H_0) \subset {\mathcal H}_{\rm loc}^2(\bar \Omega) \cap  {\mathcal H}_{0}^1(\Omega)$,  by the closed graph theorem $\partial_s  G\partial_s(H_0 -z)^{-1}  $ is bounded.  By using  similar arguments as  in the proof of the Lemma \ref{k2} $(ii)$, $ (ia( \partial_sg + g \partial_s) - ga^2) (H_{0,a}- z)^{-1}$  is also a bounded operator.

We now  conclude  the proof of the lemma by using usual arguments  \cite{reed simonIV}. 
   If  $   g_\infty ( \Im a)^2 < \lambda_0 -E_0  $, $ E_0$ remains  an  discrete eigenvalue of $H_a$ and $ e^{\Im a s} \varphi  \in L^2(\Omega)$. \qed

\subsection{ Proof of the Proposition \ref{pc1}}

 Let  $E_0$ be   a simple  eigenvalue of $H$, as above we denote by $\varphi_0$ the associated eigenvector and $P= (.,\varphi_0) \varphi_0$.\\
Let  $ \chi_1= \chi_1(s)$ be a $C^\infty$ characteristic function of the interval $[\frac{-\tau}{F}, \frac{\tau}{F}]$,  $ \tau>0$, s.t. $\chi_1(s)=1$ if $s \in [\frac{-\tau}{2F}, \frac{\tau}{2F}]$.
Introduce the following operator on $L^2(\Omega)$,
$$ H_1(F) := H + \chi_1W(F).$$
Since $n_1= \Vert \chi_1W(F) \Vert_\infty  < \infty $   then $H_1(F) $ is a selfadjoint operator on $D(H)$. Note  that $n_1= O(\tau) +O(F)$.\\
By using standard perturbation theory, we can choose  $R>0$ such that the  complex disk $\mathcal{D}=\{ z \in \mathbb C, |z-E_{0}|\leq R \} $ such that $\mathcal{D} \cap \sigma(H) = \{E_0\}$  and  has a boundary ${\partial\mathcal{D}} \subset \rho(H_1(F)$ for $ \tau$ and $F$ small enough. Then 
 \begin{equation} \label{P1}
P_1=P_1(F)=\frac{1}{2i\pi}\oint_{\partial\mathcal{D}}(H_1(F)-z)^{-1}\,dz.
\end{equation}
is  an   spectral projector  for $H_1(F)$ satisfying 
\begin{equation} \label{P11}
 \lim_{\tau \to 0, F\to 0} \Vert P_1-P\Vert \to 0.
 \end{equation}
Hence for $ \tau$ and $F$ small enough,  the operator $ H_1(F)$ has one  eigenvalues near $E_0$,  $e_0(F)$ and  $ \vert E_0-e_0(F)\vert = O(\tau) +O(F)$. Denote by $\psi_0$ the associated eigenvector. Evidently $P_1 \psi_0=\psi_0$.

 Let us show that as a consequence of the Lemma  \ref{CT},   if $F$  and $\tau$ are  small enough then  $e^{a \vert s \vert } \psi_0 \in L^2( \Omega)$ and 
 \begin{equation}\label{E} 
\Vert  e^{a \vert s \vert } \psi_0 \Vert \leq C
 \end{equation}
   where  the constant $C>0 $ and it is  independent of  $F$.\\ 
Introduce the family of operators $H_{1,a} = H_a +  \chi_1W(F)$, where $H_a$ is  defined as in the previous section. Then 
 $\{H_{1,a} , a \in \mathbb C\}$ is an entire family of type A. In the other hand  the spectrum of $H_{1,a}$ satisfies,
 $\sigma(H_{1,a}) \subset \{z \in \mathbb C,dist (z, \sigma(H_a) \leq n_1 \}$.

We have  
\begin{equation} \label{P2}
 (\varphi_0, \psi_0)  e^{as} \psi_0 =  \frac{1}{2i\pi}\oint_{\partial\mathcal{D}} e^{a s }(H_1(F)-z)^{-1} e^{-a  s }  e^{a s} \varphi_0\,dz.
 \end{equation}
For $\tau$ and $F$ small enough,   the resolvent $(H_{1,a}(F)-z)^{-1}$ is well defined for any  $ z \in \partial\mathcal{D}$. Further,  the resolvent identity
$$ (H_{1,a}(F)-z)^{-1}= (H_{a}-z)^{-1}-  (H_{a}-z)^{-1}\chi_1W(F)(H_{1,a}(F)-z)^{-1}$$
 and the fact that  $ \Vert (H_{a}-z)^{-1} \Vert $ is uniformly bounded  in $z \in  \partial\mathcal{D}$ imply that     $ \Vert(H_{1,a}(F)-z)^{-1} \Vert$    is uniformly bounded  in $z \in  \partial\mathcal{D}$ w.r.t. $\tau$ and $F$.\\
Moreover by using standard arguments,
 in the bounded operator sense $(H_{1,a}(F)-z)^{-1} =  e^{a s }(H_1(F)-z)^{-1} e^{-a  s } $ for  $z \in  \partial\mathcal{D}$.
 In the other hand we can check that    $\vert (\varphi_0, \psi_0) \vert \geq \frac{1}{2}$ if 
   $F$  and $\tau$ are  chosen small enough. 
 Hence by using  the Lemma \ref{CT} and \eqref{P2},  there exists $C>0$ independent of $\tau $ and $F$  such that 
 \begin {equation} \nonumber \Vert e^{as} \psi_0 \Vert  \leq C \Vert e^{as} \varphi_0 \Vert  < \infty.
\end{equation}
The same arguments can be applied with $a$ changing in $-a$, proving our claim. \\
From now, we fix $\tau > 0$ and  we choose $0<F<F_0$ where $F_0$ is  small enough  such that \eqref{E} also holds.

Let  $ \theta = i \beta$,  $0< \beta< \theta_0$. As in previous section $P_1=P_1(F), P_\theta= P_\theta(F)$ are  the spectral projectors of $ H_1=H_1(F), H_\theta= H_\theta (F)$  associated  respectively  to the eigenvalue $e_0, Z_0$.  We have

\begin{align}\nonumber (Z_0-e_0) (P_\theta\psi_0, P_1\psi_0)=  ((H_\theta - H_1)P_\theta\psi_0, P_1\psi_0)=& \\(( \theta    F\cos(\eta) f + (1-\chi_1) W(F) &+ \Delta T P_\theta \psi_0, P_1\psi_0)   \nonumber
\end{align}
 where  $\Delta T = T_{s,\theta} -T_s $. Hence we will use the estimate, 
 
\begin{equation} \label{diff} \vert \Im Z_0 \vert \leq \frac {1}{(\vert P_\theta\psi_0, P_1\psi_0)\vert} 
\vert ( \theta    F\cos(\eta) f + (1-\chi_1) W(F)  + \Delta T P_\theta \psi_0, P_1\psi_0)\vert  
\end{equation}
 By using  \eqref{dP},  \eqref{P11}, for $F$  and $\tau $ small enough,  the l.h.s. of \eqref{diff} is estimated as, 
 $$\vert (P_\theta\psi_0, P_1\psi_0) \vert \geq  \frac {1}{2},$$
and from   \eqref{E},   the  two first terms of the r.h.s. of \eqref{diff} satisfy
$$ \vert (\theta F\cos(\eta) fP_\theta \psi_0, \ P\psi_0) \vert \leq    \vert \theta \vert  \Vert \Phi \psi_0 \Vert = O(e^{-\frac{c}{F}})$$ 
and
 $$\vert (1-\chi_1) W(F)P_\theta \psi_0, \ P\psi_0) \vert \leq    \Vert (1-\chi_1) W(F)\psi_0 \Vert = O(e^{-\frac{c}{F}}), $$
for  some  constant $c>0$.
 Let  $\chi$ be a characteristic function of  ${\rm supp}(f')$.   Then  (see e.g  \eqref{Ts2} and \eqref{E1}),
$$ \vert ( \Delta T  P_\theta\psi_0,P \psi_0) \vert  = \vert ( \Delta T  P_\theta\psi_0, \chi P \psi_0) \vert  \leq 
\Vert \chi  \psi_0 \Vert  \Vert  \Delta T  P_\theta\psi_0\Vert . $$
Since  for $ F$  small enough, $\Vert \chi  \psi_0 \Vert  = O(e^{-\frac{c}{F}})$. Then  to prove the theorem we need  to  show that funder our conditions, $ \Vert  \Delta T  P_\theta\varphi_0\Vert$ and  then by
\eqref{Pt}  that  $ \Vert  \Delta T  (H_{\theta}(F)-z)^{-1}\Vert,\; z \in \partial\mathcal{D}$
 is   uniformly  bounded  w.r.t. 
$F$.

Note that  following the proof of the   Proposition \ref{pc}, (see e.g.
\eqref{1+k} and \eqref{DR}) then   for $F$ small enough, the norm $\Vert (H_{\theta}(F)-z)^{-1}\Vert, z \in \partial \mathcal D$ is uniformly bounded in $ F$. Evidently this is  also true  for  $\Vert (H-z)^{-1}\Vert$.
The  second resolvent  equation implies  for $F$ small and $z \in  \partial \mathcal D$,
\begin{eqnarray} \label{1}
& \Delta T  (  H_{\theta}(F)-z)^{-1} = \\& \Delta T  (  H-z)^{-1} -  \Delta T  (H-z)^{-1}\big( \Delta T  + W_\theta(F)  \big)(  H_{\theta}(F)-z)^{-1}. \nonumber
\end{eqnarray}

By the closed  graph theorem  the operator $ \Delta T(H-z)^{-1}, z \in \partial{\mathcal D}$ is bounded  and if  $F$  is assumed small enough  $ \Vert\ \Delta T  (H-z)^{-1}  \Vert <\frac{1}{2} $  uniformly in $ z \in  \partial \mathcal D$ (see e.g.  the proof of the Theorem \ref{aly ht}). In view of 
\begin {eqnarray}
& \Delta T  (H-z)^{-1} W_\theta(F) (  H_{\theta}(F)-z)^{-1} =   \\
& \Delta T   (Fs +i)(H-z)^{-1}  \frac {W_\theta(F)}{Fs  +i}  (  H_{\theta}(F)-z)^{-1} + \nonumber \\
& F \Delta T   (H-z)^{-1} ( g\partial_s + \partial_sg) (H-z)^{-1} \frac {W_\theta(F)}{Fs +i}  (  H_{\theta}(F)-z)^{-1}, \nonumber
 \end{eqnarray}
  the same arguments already used in the   Section 3,  then imply that   there exists a constant $C>0 $ such that for $F$ small enough 
 $ \Vert  \Delta T  (H-z)^{-1} W_\theta(F)  (  H_{\theta}(F)-z)^{-1} \Vert   \leq C$  for 
 $z \in  \partial \mathcal D$.
Therefore, by 
 \eqref{1}, we get  for   $ z \in  \partial \mathcal D$,
\begin{eqnarray} \label{2}
& \Vert \Delta T  (  H_{\theta}(F)-z)^{-1}  \Vert  ( 1- \Vert  \Delta T  (  H-z)^{-1}\Vert ) \leq \Vert   \Delta T (  H-z)^{-1}\Vert +  \nonumber \\ &  \Vert  \Delta T  (H-z)^{-1} W_\theta(F)(  H_{\theta}(F)-z)^{-1} \Vert, \nonumber
\end{eqnarray}
hence we get
$$ \Vert  \Delta T  (  H_{\theta}(F)-z)^{-1}  \Vert \leq   1 + 2 C. $$ 

\qed


\section {Appendix: Self-adjointness}
In this section we prove the Theorem \ref{2.1}.
Our proof     {\it }is  mainly  based  on  the commutator  theory \cite{reed simon, Rob}.  First we note that it is sufficient to show  the theorem for the operator 
$ h=h(F) = H_0 + w(F) $ defined on $L^2(\Omega)$  where $w(F)$ is defined in \eqref{w}.
  Choose $a,b \in \mathbb R^+$ such that  
$  w(F,s)  + as^2 +b >1$  and consider the positive  symmetric operator in $L^2(\Omega)$,
$$N=H_0   + w(F) + 2as^2 +b.$$
The operator N     admits a (Friedrichs) self-adjoint extension  since it is associated with a positive quadratic form, we denote its self-adjoint extension  by the same symbol \cite{Ka}.
Moreover $ N $ has compact resolvent and then only discrete spectrum (see section \ref{sec} below).  So  $N$ is essentially self-adjoint 
on 
\begin{equation} \label {Co}
 {\mathcal C}=\{\varphi=\psi_{\mid \bf\Omega}: \psi\in\mathrm{S}(\mathbb{R}^{2}), \psi(s,0) = \psi(s,d) =0\, \, \mbox{for all}\,\,s\in\mathbb{R}\} 
 \end{equation} 
where $\mathcal{S}(\mathbb{R}^{2})$ denotes the Schwartz class. 
In fact   $ {\mathcal C}$ contains  a complete set  of eigenvectors of $N$. Indeed some standard arguments (see e.g. \cite{Ag, GaYs,reed simonIV}) show that the  corresponding eigenfunctions  and their derivatives  are smooth on $\bar \Omega$ and  super-exponentially decay in the longitudinal direction. From  \cite [X.5]{reed simon}  we  have  to check  that there exist $c,d >0$ such that  for all $\varphi\in  {\mathcal C}$, $\Vert  \varphi\Vert = 1$,
  \begin{equation}
  \label{H1}
 c \|N \varphi\| \geq  \|h\varphi\|
  \end{equation}
  and 
    \begin{equation}
  \label{H2}
d\|N^{\frac{1}{2}}\varphi\|^{2} \geq  |(h\varphi,N\varphi)-(N\varphi,h\varphi)|.
  \end{equation}
In the  quadratic forms sense on $ {\mathcal C}$,
\begin{eqnarray}
N^{2} &=&(h+b)^{2}+ 4a sN s  +  [[h,s],s].
\end{eqnarray}
But in the form sense on $ {\mathcal C}$, $ [[h,s],s] =-2g $ and g is bounded function.
Therefore,  
$$ \Vert  N \varphi \Vert  +2 \Vert g\Vert_\infty \geq  \Vert (h+b)\varphi \Vert$$
 and then since $ N\geq 1$ this last inequality implies  \eqref{H1}.
Similarly,
\begin{eqnarray}
\pm i [h,N]&=&\pm i [ h-N,N] = \pm i2a[s^{2},T_s]\notag\\
&=& \mp i4a(\partial_{s}g s + s g\partial_{s})\notag,
\end{eqnarray}
this gives that for all $\varphi\in  {\mathcal C}, \Vert  \varphi\Vert = 1$, 
\begin{equation} \label{com}
|(h\varphi,N\varphi)-(N\varphi,h\varphi)|\leq 2a
( \Vert g^{\frac{1}{2}}\partial_{s}\varphi\Vert^2 + \Vert  s g^{\frac{1}{2}} \varphi \Vert^2).
\end{equation}
 Clearly  we have  $ N \geq T_s + as^2 $ on $ {\mathcal C}$. Then from  \eqref{com} there exists a constant $d>0$
 such that 
$$
|(h\varphi,N\varphi)-(N\varphi,h\varphi)|\leq d (N\varphi, \varphi)
$$
proving (\ref{H2}). 

We now show  {\it(ii)}.  Let $E\in \mathbb{R}$.   We denote by   $ \tilde E_1$  the first eigenvalue of the operator $T_u+F\sin(\eta)u$ and $\tilde \chi_{1}$ the associated normalized eigenvector, 
\begin{equation}
(T_u + F\sin(\eta)u)\tilde \chi_{1}(u)=\tilde E_{1} \tilde \chi_{1}(u).
\end{equation}
Set $ \lambda = E- \tilde E_1$ and $\varphi$ be the solution of the Airy equation 
\begin{equation}
-\varphi''(s) + F\cos(\eta) s \varphi(s)= \lambda \varphi(s)\; \; \lambda\in\mathbb{R}.
\end{equation}
It is known  (see e.g. \cite{AbSz} )   that $\varphi(s)=  (\lambda- F\cos(\eta)s)^{-1/4} e^{-i \frac{2}{3F\cos(\eta)} (\lambda- F\cos(\eta)s)^{3/2}} + o((\lambda- F\cos(\eta)s)^{-1/4})$ and  $\varphi'(s)= (\lambda- F\cos(\eta)s)^{1/4} e^{-i \frac{2}{3F\cos(\eta)} (\lambda- F\cos(\eta)s)^{3/2}} + o((\lambda- F\cos(\eta)s)^{1/4})$ as $s \to  - \infty$. \\
 Let  $\xi$ be a  $ \mathrm{C}^{\infty}$ characteristic  function of $ (-1,1)$ and  $ s \in \mathbb R \to \xi_n(s) = \xi( \frac{s}{n^{\alpha}}+ n)$,  $ \frac{1}{2}< \alpha< 1, n \in \mathbb N ^*$. Set 
 
$\psi_{n}=\frac{\widetilde{\psi}_{n}}{\|\widetilde{\psi}_{n}\|}$ where $\widetilde{\psi}_{n}(s,u)=\tilde \chi_{1}(u)\varphi(s)\xi_n(s) $, then for $n$ large enough, $\Vert \tilde  \psi_n\Vert = \Vert \varphi \xi_n \Vert \geq c\;n^{ \alpha/2-1/4} $ for some  constant
$c>0$.
 Since $g=1$ if $n$ is large, we have
\begin{eqnarray}
({H}(F)-E)\psi_{n}=\big( -2\chi_{1}(u)\varphi'(s)\xi_{n}'(s)- \chi_{1}(u)\varphi(s)\xi_{n}''(s)\big)\frac{1}{\|\widetilde{\psi}_{n}\|}\notag
\end{eqnarray}
and then 
\begin{equation}
\|({H}(F)-E)\psi_{n}\|_{L^{2}(\Omega)}\leq  \frac{1}{\|\widetilde{\psi}_{n}\|} \big(
 2\|\varphi'\xi_{n}'\|_{L^{2}(\mathbb{R})}+  \|\varphi\xi_{n}''\|_{L^{2}(\mathbb{R})} \big).
\end{equation}
For $n$ large enough $\|\varphi'\xi_{n}'\|^{2}_{L^{2}(\mathbb{R})}=o(n^{-\alpha/2 +1/4})$ 
and $\|\varphi\xi_{n}''\|_{L^{2}(\mathbb{R})}=o(n^{-3\alpha/2-1/4 })$. Thus,
\begin{equation*}
\lim_{n\rightarrow \infty }\|({H}(F)-E)\psi_{n}\|_{L^{2}(\Omega)}=0.
\end{equation*}
This completes the proof.

\qed
\subsection{The operator $(N+1)^{-1}$} \label{sec}
Consider first  the  positive self-adjoint operator  on $ L^2(\Omega)$
$$ N_0  = (-\partial_s^2 + v(s)) \otimes \mathbb I_{L^(0,d)} +   \mathbb I_{L^(\mathbb R)} \otimes T_u  $$
where $ v(s) = w(F, s) + 2as^2 +b$ and $w$ is defined in \eqref{w}. It is known that 
the operator $-\partial_s^2 + v(s)$ is essentially self-adjoint on $L^2(R)$ and has a compact resolvent \cite {reed simon, reed simonIV}. By the min-max principle we can verify that 
the eigenvalues of this operator  satisfy,  there exists $ c_1, c_2 >0 $ such that  for large 
$n \in \mathbb N$
$$ c_1n  \leq e_n \leq c_2 n. $$  
Then   $ (N_0 +1)^{-1} $ is an Hilbert-Schmidt operator.   
By using the second resolvent equation we have 
$$ (N +1)^{-1}  = (N_0 +1)^{-1}  + (N_0 +1)^{-1} \partial_s  G\partial_s (N +1)^{-1} $$
where $ G$ is defined in the proof of the Lemma \ref{l2}.
Therefore,  the statement follows if we show that $ \partial_s  G\partial_s (N +1)^{-1} $
is a bounded operator.

We have 
$$ \partial_s  G\partial_s (N +1)^{-1}  = \partial_s  G\partial_s(H_0 +1)^{-1}  - \partial_s  G\partial_s(H_0 +1)^{-1} v (N +1)^{-1}. $$
Since $D(H_0) \subset {\mathcal H}_{ \rm loc}^2(\bar \Omega) \cap  {\mathcal H}_{0}^1(\Omega)$, by the closed graph theorem $\partial_s  G\partial_s(H_0 +1)^{-1}  $ and   $\partial_s  G\partial_s s(H_0 +1)^{-1}  $  are bounded.  
 Standard commutation relations then imply,
\begin{eqnarray*}
 &\partial_s  G\partial_s(H_0 +1)^{-1}v (N +1)^{-1} = \partial_s  G\partial_s (s+i)(H_0 +1)^{-1} \frac{ v}{s+i}  (N +1)^{-1}+ \\ & \partial_s  G\partial_s(H_0 +1)^{-1}  2 \partial_s (H_0 +1)^{-1}
\frac{ v}{s+i} (N +1)^{-1}. 
\end{eqnarray*}
We know that the domain  $ D(N) \subset D( \vert v \vert^{1/2})$  so $\frac{ v}{s+i} (N +1)^{-1}$
is bounded, then  it follows by using the same arguments as above  that  $\partial_s  G\partial_s(H_0 +1)^{-1}v (N +1)^{-1}$ is also bounded.
\qed
 \bigskip 

{\bf Acknowledgments} 
M. Gharsalli  would like to thanks   the Centre de Physique Th\' eorique-CNRS  for  the warm welcome  extended  to her during   her visit and where this  present work was done.  The authors thank H.  Najar   who has pointed out to us this problem.


\newpage

\begin{thebibliography}{999}
\bibitem{AbSz} M. Abramowitz and  I.A. Stegun: {\em Handbook of mathematical functions}, National Bureau of Standards Applied Mathematics Series, 55, 1964.
\bibitem{Ad} R.A. Adams: {\em Sobolev spaces}, Academic press,  2e \' ed, 2003.
\bibitem{Ag} S. Agmon:  { \em Lectures on exponential decay of solutions of second-order elliptic equations: bounds on,
eigenfunctions of N-body Schr\"odinger operators}, Mathematical Notes, vol. 29. Princeton University
Press, Princeton, NJ, 1982.
\bibitem{AC} J.Aguilar and J.M. Combes:  {  A class of analytic perturbations for one-body Schrdinger Hamiltonians},
{ \em Comm. Math. Phys. \/}  {\bf 22}, 269 (1971). 
\bibitem{AvHe}  J.E. Avron and  I. Herbst:  { Spectral and scattering Theory of Schr\"odinger operators related to Stark effect}, { \em Comm. Math. Phys.}, {\bf 52},  239 (1977).

\bibitem{PB2}
P. Briet:  {  General estimates on distorted resolvents and application to Stark hamiltonians},  {\em Rev. Math. Phys.} {\bf 8}, no. 5,  639 (1996),  

\bibitem{BuGeReSi} W.A. Bulla, F.Gesztesy, W.Renger and  B.Simon: 
{Weakly coupled bound states in quantum waveguides},
{ \em Proc. Amer. Math. Soc.} {\bf 125}, no. 5, 1487  (1997). 
\bibitem{CFKS} H.I. Cycon, R.G.Froese, W.Kirsch and  B.Simon:
{\it Schr\" odinger Operators},
Springer-Verlag, Berlin-Heidelberg, 1987. 

\bibitem{D} E.B. Davies:  {\em Spectral theory and differential operators}, Cambridge Studies in Advanced Mathematics, 1996.

\bibitem{DE} P. Duclos and P.Exner :  {Curvature-induced bound states in quantum waveguides in two and three dimensions},  {\em Rev. Math. Phys.}, {\bf 7}, no. 1, 73  (1995).

\bibitem{Exner}
P.Exner:  {  A quantum pipette},  {\em Journal of Physics A: Mathematical and General}, {\bf  28}, Issue 18,  5323 (1995).
\bibitem{FK} C. Ferrari and  H. Kovarik: {On the exponential decay of magnetic Stark resonances}, {\em Rep. Math. Phys.}, {\bf 56}, no. 2, 197  (2005). 

\bibitem{GaYs}
J. Gagelman and H. Yserentant: {A spectral method for Schr\"odinger equations
with smooth confinement potentials},  {\em Numer. Math.} {\bf  122}, no. 2, 323  (2012), .

\bibitem{Ga}
M. Gharsalli: {Stark resonances in  a 2-dimensional curved tube II},  In preparation.

\bibitem{Ha}
E. Harrel:    {Perturbation theory and atomic resonances
since Schr\"odinger' s time}, { \em Spectral theory and mathematical physics: a Festschrift in honor of Barry Simon's 60th birthday}, {\em Proc. Sympos. Pure Math.}, {\bf 76}, Part 1, Amer. Math. Soc., Providence, RI, 2007. 

\bibitem{He}
I. Herbst:   {Dilation analyticity in constant electric field},  {\em Comm. Math. Phys.}, { \bf 64},  279 (1979).

\bibitem{HpSg}
P. Hislop and  I.M. Sigal:   {\em Introduction to spectral theory. With application to Schr\"odinger operators},
{\em Applied Mathematical Sciences}, {\bf 113}, Springer-Verlag, New York, 1996.

\bibitem{HpVi} P.Hislop and C. Villegas-Blas:  { Semiclassical Szeg\"o limit of resonance clusters for the hydrogen atom Stark hamiltonian}, {\em Asymptot. Anal.}, {\bf 79}, no. 1-2, 17  (2012). 

\bibitem{WH}
W. Hunziker:   { Distortion analyticity and molecular resonance curve}, {\em Ann. Inst. Poincar\'e},  {\bf 45},  339 (1986).

\bibitem{Ka}
T. Kato:  {\em Perturbation theory for Linear Operators},  2nd Edition, Springer Verlag, Berlin, Heilderberg,  1995.
\bibitem{Kri} J. Kriz:  {\em Spectral properties of planar quantum waveguides with combined
boundary vconditions}, P.H.D. Thesis, Charles University Prague,  2003.
\bibitem{Op}
R. Oppenheimer:  { Three notes on the quantum theory of
aperiodic effects}, {\em Phys. Rev.} {\em 31}, 66  (1928).
\bibitem{reed simon}
M. Reed and B. Simon:   {\em Methods of Modern Mathematical Physics, II.  Fourier Analysis, Self-Adjointness.},  Academic Press, London-New York,  1975.
\bibitem{reed simonIV}
M. Reed; B. Simon:  {\em Methods of Modern Mathematical Physics, IV.  Analysis of Operators} , Academic Press, New York,1978.


\bibitem{Rob}
D.W. Robinson:   { Commutator theory on Hilbert space}, {\em Can. J. Math}, {\bf 39}, N¡7,  1235 (1987).
\bibitem{Sig} I. M. Sigal: {  Geometric theory of Stark resonances in multielectron systems.} {\em Comm. Math. Phys. }{\bf 119}, no. 2,  287 (1988).




  \bibitem{Wa} X.P. Wang:  {On resonances of generalized N-body Stark hamiltonians},  {\em J. Operator Theory}, {\bf  27}, no. 1, 135 (1992). 

\end{thebibliography}
\end{document}